\documentclass[final,5p,times,twocolumn]{elsarticle}

\usepackage{amsmath,graphicx,psfrag,epsfig,color,amsfonts,subfigure}
\usepackage{amsthm,amssymb}
\usepackage{version}

\graphicspath{{./fig/}}

\def\prob{\mathbb{P}}
\def\expt{\mathbb{E}}
\def\real{\mathbb{R}}

\def\naturals{\mathbb{N}}

\newcommand{\mcN}{\ensuremath{\mathcal{N}}}

\newcommand{\until}[1]{\{1,\dots, #1\}}

\newcommand{\subscr}[2]{#1_{\textup{#2}}}

\newcommand{\setdef}[2]{\{#1 \; | \; #2\}}
\newcommand{\seqdef}[2]{\{#1\}_{#2}}

\newcommand{\map}[3]{#1: #2 \rightarrow #3}

\newcommand{\E}[1]{\ensuremath{\mathbb{E}\left[ #1 \right]}}
\newcommand{\indicator}[1]{\ensuremath{\mathcal{I}\! \left( #1 \right)}}

\usepackage[vlined,ruled]{algorithm2e}

\usepackage[normalem]{ulem} 
\usepackage{soul}

\setstcolor{red}

\newcommand\oprocendsymbol{\hbox{$\square$}}
\newcommand\oprocend{\relax\ifmmode\else\unskip\hfill\fi\oprocendsymbol}

\newcommand{\beq}{\begin{equation}}
\newcommand{\eeq}{\end{equation}}
\newcommand{\bal}{\begin{align}}
\newcommand{\eal}{\end{align}}

\newtheorem{theorem}{Theorem}
\newtheorem{lemma}[theorem]{Lemma}
\newtheorem{remark}[theorem]{Remark}

\def\bs{\boldsymbol}

\def \etal {\emph{et al.}}

\newcommand\bit[1]{\textit{\textbf{#1}}}

\begin{document}

\begin{frontmatter}

\title{Correlated Multiarmed Bandit Problem: Bayesian Algorithms and Regret Analysis \tnoteref{lab}}

\tnotetext[lab]{This research has been supported in part by ONR grant  N00014-14-1-0635, ARO grant W911NF-14-1-0431 and NSF grant ECCS-1135724. }

\author[label1]{Vaibhav Srivastava}
\author[label2]{Paul Reverdy}
\author[label1]{Naomi Ehrich Leonard}

\address[label1]{Department of Mechanical \& Aerospace Engineering, Princeton University, New Jersey, USA, \emph{\texttt{\{vaibhavs, naomi\}@princeton.edu}}}
\address[label2]{Department of Electrical and Systems Engineering, University of Pennsylvania, Pennsylvania, USA \emph{ \texttt{preverdy@seas.upenn.edu}}}

\begin{abstract}
We consider the correlated multiarmed bandit (MAB) problem in which the rewards associated with each arm are modeled by a multivariate Gaussian random variable, and we investigate the influence of the assumptions in the Bayesian prior on the performance of  the upper credible limit (UCL) algorithm and a new correlated UCL algorithm. We rigorously characterize the influence of accuracy, confidence, and correlation scale in the prior on the decision-making performance of the algorithms. 
Our results show how priors and correlation structure can be leveraged to improve performance. 
\end{abstract}


\end{frontmatter}

\section{Introduction}

MAB problems~\cite{JG-KG-RW:11} are a class of resource allocation problems in which a decision-maker allocates a single resource by sequentially choosing one among a set of competing alternative options called arms. In the so-called stationary MAB problem, a decision-maker at each discrete time instant chooses an arm and collects a reward drawn from an unknown stationary probability distribution associated with the selected arm. The objective of the decision-maker is to maximize the total expected reward aggregated over the sequential allocation process. These problems capture the fundamental trade-off between exploration (collecting more information to reduce uncertainty) and exploitation (using the current information to maximize the immediate reward), and they model a variety of robotic missions including search and surveillance.

Recently, there has been significant interest in Bayesian algorithms for the MAB problem~\cite{EK-OC-AG:12,NS-AK-SMK-MS:12,SA-NG:12, PR-VS-NEL:13d}. Bayesian methods are attractive because they allow for incorporating prior knowledge and spatial structure of the problem through the prior in the inference process.  

In this paper, we investigate the influence of the prior on the performance of a Bayesian algorithm for the MAB problem with Gaussian rewards. 



MAB problems became popular following the seminal paper by Robbins~\cite{HR:52} and gathered interest in diverse areas including controls~\cite{RA-MVH-DT:88, VA-PV-JW:87}, robotics~\cite{JLN-MD-EF:08,MYC-JL-FSH:13,VS-PR-NEL:14a}, machine learning~\cite{MB-YS-AS:09, FR-RK-TJ:08}, economics~\cite{BPM-JJM:87}, ecology~\cite{JRK-AK-PT:78,VS-PR-NEL:13e}, and neuroscience~\cite{RCW-etal:14, PR-VS-RCW-NEL:15}. 
Much recent work on  MAB problems focuses on a quantity termed \emph{cumulative expected regret}. 
The cumulative expected regret of a sequence of decisions is the cumulative difference between the expected reward of the options chosen and the maximum possible expected reward. 
In a ground-breaking work, Lai and Robbins~\cite{TLL-HR:85} established a logarithmic lower bound on the expected number of times a sub-optimal arm needs to be sampled by an optimal policy in a frequentist setting, thereby showing that cumulative expected regret is bounded below by a logarithmic function of time. Their work established the best possible performance of any solution to the standard MAB problem.  They also developed an algorithm based on an upper confidence bound on estimated reward and showed that this algorithm  achieves the performance bound asymptotically. 

In the following, we use the phrase \emph{logarithmic regret} to refer to cumulative expected regret being bounded above by a logarithmic function of time, i.e., having the same order of growth rate as the optimal solution. 


In the context of the bounded MAB problem, i.e., the MAB problem in which the reward is sampled from a distribution with a bounded support,  Auer~\etal~\cite{PA-NCB-PF:02} developed upper confidence bound-based algorithms that achieve logarithmic regret uniformly in time; see~\cite{SB-NCB:12} for an extensive survey of upper confidence bound-based algorithms. 



Bayesian approaches to the MAB problem have also been considered. Srinivas~\etal~\cite{NS-AK-SMK-MS:12} developed asymptotically optimal upper confidence bound-based algorithms for Gaussian process optimization. 
Agrawal~and~Goyal~\cite{SA-NG:12, EK-NK-RM:12} showed that a Bayesian algorithm known as Thompson sampling~\cite{WRT:33} is near-optimal for binary bandits with a uniform prior. Liu~and~Li~\cite{CYL-LL:15} characterize the sensitivity of the performance of Thompson sampling to the assumptions on prior. 
Kaufman~\etal~\cite{EK-OC-AG:12} developed a generic Bayesian upper confidence bound-based algorithm and established its optimality for binary bandits with a uniform prior. 

Reverdy~\etal~\cite{PR-VS-NEL:13d} studied the Bayesian algorithm proposed in~\cite{EK-OC-AG:12} in the case of correlated Gaussian rewards and analyzed its performance for uninformative priors. They called this algorithm the upper credible limit (UCL) algorithm and showed that the UCL algorithm models human decision-making in the spatially-embedded MAB problem. We define a spatially-embedded MAB problem as an MAB problem in which the arms are embedded in a metric space and the correlation coefficient between arms is a function of distance between them. For example, in the problem of spatial search over an uncertain distributed resource field, patches in the environment can be modeled as spatially located alternatives and the spatial structure of the resource distribution as a prior on the spatially correlated reward. This is an example of a \emph{spatially-embedded MAB problem.} 
It was observed in~\cite{PR-VS-NEL:13d} that good assumptions on the correlation structure result in significant improvement of the performance of the UCL algorithm, and these assumptions can successfully account for the better performance of human subjects. 

In this note we rigorously study the influence of the assumptions in the prior on the performance of the UCL algorithm  for a MAB problem with Gaussian rewards.   Since the UCL algorithm models human decision-making well, the results in this paper help us identify the set of parameters in the prior that explain the individual differences in performance of human subjects.
The major contributions of this work are twofold:

First, we study the UCL algorithm with uncorrelated informative prior and characterize its performance. We illuminate the opposing influences 
of  the degree of confidence of a prior and the magnitude of its inaccuracy, i.e., the gap between its mean prediction and the true mean reward value, on the decision-making performance.  
 
Second, we propose and study a new correlated UCL algorithm with correlated informative prior and characterize its performance.  We show that large correlation scales reduce the number of steps required to explore the surface. We then show that incorrectly assumed large correlation scales may lead to a much higher number of selections of suboptimal arms than suggested by the Lai-Robbins bound. This analysis provides insight into the structure of good priors in the context of explore-exploit problems. 

The remainder of the paper is organized in the following way. In Section~\ref{sec:review}, we recall  the MAB problem and an associated Bayesian algorithm, UCL. We analyze the UCL algorithm for uncorrelated informative prior and correlated informative prior in Section~\ref{sec:uncorr-UCL}~and~\ref{sec:corr-UCL}, respectively. We illustrate our results with some numerical examples in Section~\ref{sec:numerics}, and  we conclude in Section~\ref{sec:conclusions}.

\section{MAB Problem and Bayes-UCB Algorithm}\label{sec:review}

In this section we recall the MAB problem and the Bayes-UCB algorithm proposed in~\cite{EK-OC-AG:12}. 

\subsection{The MAB problem}
The $N$-armed bandit problem refers to the choice among $N$ options that a decision-making agent should make to maximize the cumulative expected reward.
The agent collects reward $r_t\in \real$ by choosing arm $i_t$ at each time $t \in \until{T}$, where $T\in \naturals$ is the horizon length for the sequential decision process.  In the so-called stationary MAB problem, the reward from option  $i \in \until{N}$  is sampled from a stationary distribution $p_i$ and has an unknown mean $m_i \in \real$.     
The decision-maker's objective is to maximize the cumulative expected reward $\sum_{t=1}^T m_{i_t}$ by selecting a sequence of arms $\seqdef{i_t}{t\in\until{T}}$.
Equivalently, defining $m_{i^*} = \max \setdef{m_i}{i\in\until{N}}$ and $R_t = m_{i^*}-m_{i_t}$ as the expected \emph{regret} at time $t$, the objective can be formulated as minimizing the cumulative expected regret defined by
\begin{align*} 
\sum_{t=1}^T R_t = Tm_{i^*} - \sum_{i=1}^N m_i \E{n_{i}(T)}= \sum_{i=1}^N \Delta_i \E{n_{i}(T)},
\end{align*}
where $n_{i}(T)$ is the total number of times option $i$ has been chosen until time $T$ and $\Delta_i = m_{i^*}-m_i$ is the expected regret due to picking arm $i$ instead of arm $i^*$.

\subsection{The Bayes-UCB algorithm}

The Bayes-UCB algorithm for the stationary $N$-armed bandit problem was proposed in~\cite{EK-OC-AG:12}. The Bayes-UCB algorithm at each time
\begin{enumerate}
\item computes the posterior distribution of the mean reward at each arm;
\item computes a $(1-\alpha(t))$ upper credible limit for each arm;
\item selects the arm with highest upper credible limit.
\end{enumerate}
In step (ii), the upper credible limit is defined as the least upper bound to the upper credible set, and the function $\map{\alpha}{\naturals}{(0,1)}$ is tuned to achieve efficient performance. In the context of Bernoulli rewards, Kaufmann~\etal~\cite{EK-OC-AG:12} set $\alpha(t) = 1/(t(\log T)^c)$, for some $c\in \real_{\ge 0}$, and show that for $c \ge 5$ and uninformative priors, the Bayes-UCB algorithm achieves the optimal performance. 

Reverdy~\etal~\cite{PR-VS-NEL:13d, PR-VS-RCW-NEL:15} studied the Bayes-UCB algorithm in the context of Gaussian rewards with known variances. For simplicity the algorithm in~\cite{PR-VS-NEL:13d, PR-VS-RCW-NEL:15}  is called the  UCL (upper credible limit) algorithm. It is shown that for an uninformative prior, the UCL algorithm is order-optimal, i.e., it achieves cumulative expected regret that is within a constant factor of that suggested by the Lai-Robbins bound.  
It is also shown  that a variation  of the UCL algorithm models human decision-making  in an MAB task.

\section{Uncorrelated Gaussian MAB Problem}\label{sec:uncorr-UCL}

In this paper, we focus on the Gaussian MAB problem, i.e.,  the reward distribution $p_i$ is Gaussian with mean $m_i$ and variance $\sigma_s^2$. The  variance $\sigma_s^2$ is assumed known, e.g., from previous observations or known characteristics of the reward generation process. We now recall the UCL algorithm and analyze its performance for a general prior.

\subsection{The UCL algorithm}\label{subsec:deterministic-ucl}
Suppose the prior on the mean rewards at each arm is a Gaussian random variable with mean vector $\mu_i^0 \in \real$ and variance $\sigma_0^2  \in \real_{>0}, i\in \until{N}$.

For the above MAB problem,  let the number of times arm $i$ has been selected until time $t$ be denoted by $n_i(t)$.  Let the empirical mean of the rewards from arm $i$ until time $t$ be $\bar m_i(t)$. Then, the posterior distribution at time $t$ of the mean reward at arm $i$ has mean and variance 
\begin{align*}
\mu_{i}(t)&= \frac{\delta^2 \mu_{i}^0 + n_i(t) \bar m_i(t)}{\delta^2+n_{i}(t)}, \; \text{and}\;
\sigma_i^2 (t) =\frac{\sigma_s^2}{\delta^2 + n_{i}(t)} ,
\end{align*}
respectively, where $\delta^2=\sigma_s^2/\sigma_0^2$. Moreover,
\begin{align*}
\expt[\mu_{i}(t)]& = \frac{\delta^2 \mu_{i}^0 + n_i(t) m_i}{\delta^2+n_{i}(t)} \; \text{and}\; \text{Var}[\mu_{i}(t)] = \frac{ n_i(t) \sigma_s^2}{(\delta^2+n_{i}(t))^2}.
\end{align*}

The UCL algorithm for the Gaussian MAB problem, at each decision instance $t \in \until{T}$,  selects an arm  with the maximum $(1-1/Kt)$-upper credible limit, i.e., it selects an arm $i_t=\text{argmax}\setdef{Q_i(t)}{i\in\until{N}}$, where
\[
Q_i(t) =  \mu_i(t) + \sigma_i(t)  \Phi^{-1}(1-\alpha_t).
\]
$\map{\Phi^{-1}}{(0,1)}{\real}$ is the inverse cumulative distribution function for the standard Gaussian random variable, $\alpha_t = 1/Kt^a$, and $K \in \real_{>0}$ and $a\in \real_{>0}$  are tunable parameters.

In the context of Gaussian rewards, the function $Q_i(t)$  decomposes into two terms corresponding to the estimate of the mean reward and the associated variance. This makes the UCL algorithm amenable to an analysis akin to the analysis for UCB1~\cite{PA-NCB-PF:02}. Using such an analysis,
it was shown in~\cite{PR-VS-NEL:13d} that the UCL algorithm with an uninformative prior and parameter values $K =\sqrt{2\pi e}$ and $a=1$ achieves an order-optimal performance. In the following, we investigate the performance of the UCL algorithm for general priors.

\subsection{Regret Analysis for uncorrelated prior}\label{sec:uncorr-UCL-regret}

To analyze the regret of the  UCL algorithm, we require some inequalities that we recall in the following lemma.

\begin{lemma}[\bit{Relevant inequalities}]\label{lem:ineq}
For  the standard normal random variable $z$ and the associated inverse cumulative distribution function $\Phi^{-1}$, the following statements hold:
\begin{enumerate}
\item for any $w\in [0, +\infty)$
\begin{align*}
 \prob(z\ge w) &\le \frac{2 e^{-w^2/2}}{\sqrt{2 \pi} ( w + \sqrt{w^2 + 8/\pi})} \le \frac{1}{2} e^{-w^2/2}\\ 
  \prob(z\ge w) &\ge \sqrt{\frac{2}{\pi}} \frac{e^{-w^2/2}}{w + \sqrt{w^2 +4}};
\end{align*}
\item for any $\alpha \in [0, 0.5]$, $t\in \naturals$ and $a>1$, 
\begin{align*}
 \Phi^{-1}(1- \alpha) & \le \sqrt{-2 \log (\alpha)} \\
\Phi^{-1}(1-\alpha)  & > \sqrt{-\log(2\pi \alpha^2(1-\log(2\pi \alpha^2)))}\\
\Phi^{-1}\Big(1- \frac{1}{\sqrt{2 \pi e} t^a}\Big) &> \sqrt{\frac{3a}{2} \log t}.
\end{align*}
\end{enumerate}
\end{lemma}
Statement (i) in Lemma~\ref{lem:ineq} can be found in~\cite{MA-IAS:64}. The first inequality in (ii) follows from (i). The second inequality in (ii) was established in~\cite{PR-VS-NEL:13d}, and the last inequality can be easily verified using the second inequality in (ii). 


\begin{lemma}[\bit{Difference of squares inequality}]\label{lem:diff-of-squares}
For any $c_1, c_2 \in \real$ such that $ (1-c_1)(1+c_2) \ge 1$, 
\begin{equation*}
(x-y)^2 \ge c_1 x^2 - c_2 y^2, \quad \text{for any} \quad x, y \in \real.
\end{equation*}
\end{lemma}
\begin{proof}
The inequality follows trivially using a completing the square argument.
\end{proof}

Let $\Delta m_{i} = m_{i}-\mu_{i}^0$, for each $i\in \until{N}$. Set $a >\frac{4}{3} (1 + \frac{\delta^2}{1-\epsilon})$,  $c_1= \frac{1-\epsilon}{1+\delta^2 -\epsilon}$, and 
$c_2 = \frac{1-\epsilon}{\delta^2}$, for some $\epsilon\in (0, 1)$.

\begin{theorem}[\bit{Regret for uncorrelated prior}]\label{thm:uncorr-regret}
For the Gaussian MAB problem, and the UCL algorithm with uncorrelated prior, the expected number of times a suboptimal arm $i$ is selected satisfies
\[
\expt[n_i(T)] \le  \eta_i + \hat n_i(T), 
\] 
where $\eta_i = \max\{1, \lceil  \frac{4 \sigma_s^2}{\Delta_i^2}(2 \log K + 2a \log T) - \delta^2\rceil \} $, and $\hat n_i(T)$ is defined in~\eqref{eq:subopt-sel-extra}.
\begin{figure*}[bth!]
\begin{equation}\label{eq:subopt-sel-extra}
\hat n_i(T) = \begin{cases}
\max \Big\{ e^{\frac{2 \delta^2 \Delta m_{i^*}^2}{3 a \sigma_0^2 }}, e^{\frac{2 \Delta m_{i^*}^2}{3 a \sigma_0^2}} \Big\} + \frac{3 a c_1}{2(3 a c_1 -4)} e^{ \frac{c_2 \delta^2 \Delta m_{i^*}^2}{2\sigma_0^2}}
+ e^{\frac{2 \delta^2 \Delta m_{i}^2}{3 a \sigma_0^2 \eta_i}}  +  \frac{3 a c_1}{2(3 a c_1 -4)}e^{\frac{c_2 \delta^2 \Delta m_{i}^2}{ 2\sigma_0^2 \eta_i} },
& \text{if } \Delta m_{i^*} >0, \Delta m_{i} < 0, \\
\max \Big\{ e^{\frac{2 \delta^2 \Delta m_{i^*}^2}{3 a \sigma_0^2 }}, e^{\frac{2 \Delta m_{i^*}^2}{3 a \sigma_0^2}} \Big\} + \frac{3 a c_1}{2(3 a c_1 -4)} e^{ \frac{c_2 \delta^2 \Delta m_{i^*}^2}{2\sigma_0^2}}
+ \frac{a}{K(a-1)},
& \text{if } \Delta m_{i^*} >0, \Delta m_{i} \ge 0, \\
 e^{\frac{2 \delta^2 \Delta m_{i}^2}{3 a \sigma_0^2 \eta_i}}  +  \frac{3 a c_1}{2(3 a c_1 -4)}e^{\frac{c_2 \delta^2 \Delta m_{i}^2}{ 2\sigma_0^2 \eta_i} } + \frac{a}{K(a-1)},
& \text{if } \Delta m_{i^*} \le 0, \Delta m_{i} < 0,\\
\frac{2 a}{K(a-1)},
& \text{if } \Delta m_{i^*} \le 0, \Delta m_{i} \ge  0.
\end{cases}
\end{equation}
\hrule
\end{figure*}
\end{theorem}
\begin{proof}
See~\ref{proof-uncorr-regret}.
\end{proof}

\begin{remark}[\bit{Regret of uncorrelated UCL algorithm}] \label{rem:uncor-regret}
\textup{ The expression for $\hat{n}_i(t)$ in~\eqref{eq:subopt-sel-extra} suggests that if the prior underestimates a suboptimal arm and overestimates the optimal arm, then $\hat{n}_i(t)$ is a small constant (the last case in~\eqref{eq:subopt-sel-extra}). Further, if $\sigma_0^2$ is small, i.e., the prior is confident in these estimates, then a large constant $\delta^2$ is subtracted from the logarithmic term in $\eta_i$ defined in Theorem~\ref{thm:uncorr-regret}. This leads to a substantially smaller expected number of suboptimal selections $\expt[n_i(T)]$ for an informative prior compared to an uninformative prior  over a short time horizon.}

\textup{If the prior underestimates the optimal arm which corresponds to the first two cases in~\eqref{eq:subopt-sel-extra},  then $\hat n_i(T)$ is a large constant that depends exponentially on $\Delta m_{i^*}^2 / \sigma_0^2$ . A similar effect is observed if a suboptimal arm is overestimated which corresponds to the first and third case in~\eqref{eq:subopt-sel-extra}. 
Further, if $\sigma_0^2$ is small, then the reduction in expected number of suboptimal selections  due to large $\delta^2$ in $\eta_i$ may be overpowered by the large constant in $\hat n_i(T)$. 
Here, there exists a range of $\sigma_0$, for which an informative prior leads to a smaller expected number of suboptimal selections $\expt[n_i(T)]$ over short time horizon compared to an uninformative prior. }

\textup{
In the asymptotic limit $T \to +\infty$, the logarithmic term in $\eta_i$ dominates and both informative and uninformative priors will lead to a similar performance.} \oprocend
\end{remark}

\section{Correlated Gaussian MAB problem}\label{sec:corr-UCL}

In this section, we study a new correlated UCL algorithm for the correlated MAB problem. We first propose a modified  UCL algorithm, and then analyze its performance. The modification is designed to leverage prior information on correlation structure.

\subsection{The correlated UCL algorithm}\label{subsec:deterministic-ucl-corr}
Suppose the prior on the mean rewards at each arm is a multivariate Gaussian random variable with mean vector $\bs \mu_0 \in \real^N$ and covariance matrix $\Sigma_0  \in \real^{N\times N}$. 

For the above MAB problem, the posterior distribution of the mean rewards at each arm at time $t$ is a Gaussian distribution with mean $\bs \mu (t)$ and covariance $\Sigma(t)$ defined by
\begin{equation} \label{eq:correlated-inference}
\begin{array}{ll}
\bs q(t) & = \frac{r(t) \bs \phi(t)}{\sigma_s^2} + \Lambda(t-1) \bs \mu(t-1)\\
\Lambda(t) &= \frac{\bs \phi(t) \bs \phi(t)^T}{\sigma_s^2} + \Lambda(t-1), \quad \Sigma(t) = \Lambda(t)^{-1}\\
\bs \mu(t) &= \Sigma(t) \bs q(t),
\end{array}
\end{equation}
where $\bs \phi(t)$ is the column $N$-vector with $i_t$-th entry equal to one, and every other entry zero. 
In the following, we denote entries of $\bs \mu(t)$ and the diagonal entries 
of $\Sigma(t)$ by $\mu_i(t)$ and $\sigma_i^2(t), i\in \until{N}$, respectively.

As in Section~\ref{subsec:deterministic-ucl}, let  $n_i(t)$ be the number of times arm $i$ has been selected until time $t$, and $\bar m_i(t)$ be the empirical mean of the rewards from arm $i$ until time $t$. Then, it is easy to verify that 
\begin{align}\label{eq:inference-corr-simplified}
\begin{array}{ll}
\bs \mu(t) & = (\Lambda_0 + P(t)^{-1})^{-1} (P(t)^{-1} \bar{\bs{m}}(t) +  \Lambda_0 \bs \mu_0) \\
\Lambda(t) & = \Lambda_0 + P(t)^{-1},
\end{array}
\end{align}
where $\Lambda_0 = \Sigma_0^{-1}$, $P(t) $ is the diagonal matrix with entries $\sigma_s^2/n_i^t, \; i\in \until{N}$, and $\bar{\bs m}(t)$ is the vector of $\bar m_i(t), i\in\until{N}$.

The correlated UCL algorithm for the Gaussian MAB problem, at each decision instance $t \in \until{T}$,  selects an arm  with the maximum upper credible limit, i.e., it selects an arm $i_t=\text{argmax}\setdef{Q_i(t)}{i\in\until{N}}$, where
\[
Q_i(t) =  \mu_i(t) + \sigma_i(t) \sqrt{\sum_{j=1}^N {\rho_{ij}^2(t)}} \Phi^{-1}(1-\alpha_t),
\]
$\map{\Phi^{-1}}{(0,1)}{\real}$ is the inverse cumulative distribution function for the standard Gaussian random variable, $\alpha_t = 1/Kt^a$, $\rho_{ij}(t)$ is the correlation coefficient between arm $i$ and arm $j$ at time $t$ and $K \in \real_{>0}$ and $a\in \real_{>0}$  are tunable parameters. Note that for uncorrelated priors, $\sum_{j=1}^N \rho_{ij}^2(t) =1$ and the correlated UCL algorithm reduces to the UCL algorithm. 

In the context of uninformative priors, $Q_i(1) = +\infty$ for each $i\in \until{N}$, and the UCL algorithm selects each arm once in first $N$ steps. In a similar vein, we introduce an initialization phase for the correlated UCL algorithm.  

\noindent
{\bit{Initialization:}}
In the initialization phase, an arm $i_t$ defined by 
\[
i_t = \textrm{argmax}\setdef{\sigma_i^2(t-1)}{\sigma_i^2(t-1) > \sigma_s^2/ \nu, \; \text{and } i \in \until{N}},
\]
is selected at time $t$. Here, $\nu \le 1$ is a pre-specified positive constant.  Let $\subscr{t}{init}$ be the number of steps in the initialization phase.

%
\begin{lemma}[\bit{Initialization Phase}] \label{lem:initialization-corr-stat}
For the correlated MAB problem and the inference process~\eqref{eq:correlated-inference}, the initialization phase ends in at most $N$ steps and 
 the variance following the initialization phase $\sigma_i^2(\subscr{t}{init}) \le  \sigma_s^2/ \nu$, for each $i \in \until{N}$. 
\end{lemma}
\begin{proof}
Note that to prove the lemma, it suffices to show that no arm will be selected twice in the initialization phase. 

It follows from the Sherman-Morrison formula for the rank-$1$ update for the covariance in~\eqref{eq:correlated-inference} that 
\begin{equation}\label{eq:sher-morr}
\sigma_i^2(t) = \sigma_i^2(t-1) - \frac{\sigma_{i i_t}^2(t-1)}{\sigma_s^2 + \sigma_{i_t}^2(t-1)},
\end{equation}
where $\sigma_{ij}^2(t)$ is the $i,j$ component of $\Sigma(t)$, for each $i \in \until{N}$. If $i_t= j$, then $\sigma_{j}^2(t) =  \frac{\sigma_{j}^2(t-1) \sigma_s^2}{ \sigma_{j}^2(t-1) + \sigma_s^2} \le \sigma_s^2$.  Thus, arm $j$ will not be selected again in the initialization phase which establishes our claim. 
\end{proof}

\begin{remark}[\bit{Correlation Structure and Initialization}]\label{rem:corr-init}\textup{
Lemma~\ref{lem:initialization-corr-stat} states that the length of the initialization phase is upper bounded by $N$. 
For an uninformative prior, the above initialization phase reduces to visiting each arm once, and the variance at each arm after the initialization phase is $\sigma_s^2$ ($\nu = 1$). 
In this case, the upper bound $N$ on the number of steps in the initialization phase is achieved. 
For an informative prior with correlation structure, the initialization phase may be shorter than $N$ steps, i.e., not all arms need to be visited.  This is because a visit to one arm may reduce variance in correlated arms even if unvisited.  However,  the variance at those arms not visited during the initialization phase might still be greater than $\sigma_s^2$, i.e., the bound in Lemma~\ref{lem:initialization-corr-stat} will be met but it is possible that $\nu < 1$. 
To see how variance can be reduced in arms not visited, 
note the effect of prior covariance $\sigma_{i i_t}^2(t-1)$ on the reduction in variance of an arm $i \ne i_t$.
In particular, it follows from~\eqref{eq:sher-morr} that $\sigma_i^2(t) = \frac{\sigma_s^2 \sigma_i^2(t-1) - \sigma_i^2(t-1) \sigma_{i_t}^2(t-1) (1 - \rho^2_{i i_t}(t-1))}{\sigma_s^2 + \sigma_{i_t}^2(t-1)} $. Thus, a high value of correlation $\rho_{i i_t}(t-1)$ leads to substantial reduction in variance of arm $i$ even when it is not selected. 
}

\textup{
To better understand the role of correlation, consider a set of arms comprised of decoupled clusters of highly correlated arms. Consider such a cluster of arms with cardinality $m$. The initial covariance matrix for this cluster is $\sigma_0^2 (\bs 1_m \bs 1_m^\top + \varepsilon E)$, where $E$ is a symmetric perturbation matrix with zero diagonal entries, $\bs 1_m$ is the vector of length $m$ with all entries equal to one, and $ 0 < \varepsilon  \ll  1$. It follows that one eigenvalue of $\sigma_0^2 (\bs 1_m \bs 1_m^\top + \varepsilon E)$ is $\sigma_0^2 m + O(\sigma_0^2 \varepsilon)$ and other eigenvalues are $O(\sigma_0^2 \varepsilon)$. In this setting, just one sample can significantly reduce the eigenvalue at $\sigma_0^2 m + O(\sigma_0^2 \varepsilon)$. Since the largest eigenvalue of the covariance matrix is an upper bound on the variances, just one sample
will reduce the uncertainty associated with the cluster substantially. Thus, in the initialization phase, we need a number of observations equal to the number of clusters, which may be substantially smaller than the number of arms.  }

\textup{
It should also be noted that correlation plays a role only for short time horizons. Once each arm as been sampled sufficiently, then the matrix $\Lambda(t)$ in~\eqref{eq:inference-corr-simplified} is substantially diagonally dominant and behaves like a diagonal matrix. \oprocend
} 
\end{remark}




\subsection{Regret analysis for correlated UCL algorithm}
For correlated priors, the inference equations~\eqref{eq:inference-corr-simplified} yield the following expressions for the bias $\bs e$ and covariance $\bar \Sigma$ of the estimate $\bs \mu(t)$
\begin{align*}
\bs e(t) &:= \expt[\bs \mu_t] - \bs m  =  (\Lambda_0 + P(t)^{-1})^{-1} \Lambda_0 (\bs \mu_0 - \bs m)\\
\bar \Sigma(t) &:=\text{Cov}(\bs \mu_t)  = (\Lambda_0 + P(t)^{-1})^{-1} P(t)^{-1}(\Lambda_0 + P(t)^{-1})^{-1},
\end{align*}
where $\bs m$ is the vector of mean reward. 

Let $\sigma_i^2(t)$ and $\sigma_{ij}(t)$, $i,j\in\until{N}$ be the diagonal and off-diagonal entries of $\Sigma(t)$, and $\bar \sigma_i^2(t), i\in \until{N}$ be the diagonal entries of $\bar \Sigma(t)$. 

%
%

We now analyze the properties of covariance matrices $\Sigma(t)$ and $\bar \Sigma(t)$. 
Let $\Sigma_{\sim i}(0) \in \real^{(N-1) \times (N-1)}$ be the submatrix of $\Sigma_0$ obtained after excluding the $i$-th row and $i$-th column. Let $\sigma_{i}(0) \in \real^{N-1}$ be the row vector obtained after excluding the $i$-th entry from the $i$-th row of $\Sigma_0$. We define the variance of arm $i$ conditioned on the mean reward at every other arm by
\[
\subscr{\sigma}{$i$-cond}^2 = \sigma_i^2(0) - \sigma_{i}(0) \Sigma_{\sim i}^{-1}(0)\sigma_{i}^{\top}(0). 
\]
Let $\subscr{\delta}{$i$-cond}^2 = \sigma_s^2 /\subscr{\sigma}{$i$-cond}^2$.
With a slight abuse of notation, we refer to $n_i(t)$ as the number of times arm $i$ is selected after the initialization phase. We also define for each $i \in \until{N}$
\[
\beta_i =  \sqrt{\frac{ \sigma_s^2(1+ \subscr{\delta}{$i$-cond}^2)} {\nu}} \sum_{j=1}^N \sum_{k=1}^N  |\lambda_{kj}^0| |\mu_0^j - m_j| ,
\]
where $\lambda_{kj}^0$ is the $k,j$ component of $\Lambda_0$. 

\begin{lemma}[\bit{Bounds on variances}]\label{lem:bounds-variances}
The following statements hold for the inference process~\eqref{eq:correlated-inference}:
\begin{enumerate}
\item the variance $\sigma_i^2(t)$ satisfies 
\begin{align*}
\sigma_i^2(t) & \le \frac{\sigma_s^2 }{ \nu + n_i(t)}, \text{ and}\\
\sigma_i^2(t) & \ge \frac{\sigma_s^2}{\subscr{\delta}{$i$-cond}^2 + n_i(t)};
\end{align*}
\item the variance $\bar \sigma_i^2(t)$ satisfies
\begin{align*}
\bar \sigma_i^2(t) & \le \sigma_i^2(t) \sum_{j=1}^N \rho_{ij}^2(t), \text{ and}\\
\bar \sigma_i^2(t) & \ge \frac{n_i(t) \sigma_i^4(t)}{\sigma_s^2}.
\end{align*}
\end{enumerate}
\end{lemma} 
\begin{proof}
We start by establishing the first statement. 
The covariance update in \eqref{eq:correlated-inference} can be simplified using the Sherman-Morrison formula to obtain
\begin{align}\label{eq:sherman-morrison}
\Sigma(t+1) = \Sigma(t) - \frac{\Sigma(t) \phi_t \phi_t^\top \Sigma(t)}{\sigma_s^2 + \phi(t+1)^{\top} \Sigma(t) \phi(t+1)}.
\end{align}
It follows that
\[
\sigma_i^2(t+1) = \sigma_i^2(t) - \frac{\sigma_{i i_t}^2(t)}{\sigma_s^2 + \sigma_{i_t}^2(t)}.
\]
It follows that after the initialization phase $\sigma_i^2(t) \le \nu$. 
Moreover, at each future round, if $i_t \ne i$, then $\sigma_i^2(t+1) \le \sigma_i^2(t)$; otherwise, $\sigma_i^2(t+1) = \sigma_s^2 \sigma_i^2(t)/(\sigma_s^2 + \sigma_i^2(t))$. The upper bound on $\sigma_i^2(t)$ immediately follows from this observation and the induction argument. 

We now establish the lower bound on $\sigma_i^2(t)$. 
Since the inference process involves a stationary environment, the sequence in which arms are played is of no significance and the inference only depends on the number of times an arm has been played. Consequently, the inference is the same if arms are played in blocks. In particular, each arm $j \in \until{N}$ can be played in a block of size $n_j(t)$. Further, any order in which these blocks are played leads to the same inference. 

Suppose for such a modified allocation of arms, $t_j$ is the time when the block associated with arm $j$ begins. Suppose that arm $i$ is played the last.  Then, from~\eqref{eq:sherman-morrison} and for the modified allocation process, it follows that
\begin{align*}
\sigma_{i}^2(t_j+n_j(t)) & = \sigma_{i}^2(t_j) - \frac{n_j(t) \sigma_{ij}^2 (t_j)}{\sigma_s^2 + n_j(t) \sigma_{j}^2(t_j)}\\
& \ge \sigma_{i}^2(t_j) - \frac{\sigma_{ij}^2 (t_j)}{\sigma_{j}^2(t_j)},
\end{align*}
%
%
%
i.e., the posterior variance $\sigma_i^2(t_j + n_j(t))$ is lower bounded by the conditional variance of arm $i$ under a noise free reward from arm $j$. It follows that, for the modified allocation sequence, $\sigma_i^2(t- n_i(t)) \ge \subscr{\sigma}{$i$-cond}^2$. Now, the lower bound follows from the variance update after the last block. 

To establish the second statement, we note that $\bar \Sigma(t) = \Sigma(t) P(t)^{-1} \Sigma(t)$. It follows that
\begin{align*}
\bar \sigma_i^2(t) & = \sum_{j=1}^N  \frac{n_j(t) \sigma_{ij}^2(t)}{\sigma_s^2} 
\le \sigma_i^2(t) \sum_{j=1}^N  \frac{n_j(t) \sigma_j^2(t) \rho_{ij}^2(t)}{\sigma_s^2}\\
& \le  \sigma_i^2(t) \sum_{j=1}^N  \frac{n_j(t) \rho_{ij}^2(t)}{n_j(t) + \nu} \le \sigma_i^2(t) \sum_{j=1}^N  \rho_{ij}^2(t) ,
\end{align*}
where the second inequality follows from the fact $\sigma_j^2(t) \le \sigma_s^2/(n_j(t)+ \nu)$.

Similarly, 
\[
\bar \sigma_i^2(t)  = \sum_{j=1}^N  \frac{n_j(t) \sigma_{ij}^2(t)}{\sigma_s^2} \ge  \frac{n_i(t) \sigma_i^4(t)}{\sigma_s^2},
\]
establishing the lower bound.
\end{proof}


\begin{theorem}[\bit{Regret of correlated UCL algorithm}]\label{thm:corr-regret}
For the Gaussian MAB problem, and the correlated UCL algorithm, the expected number of times a suboptimal arm $i$ is selected after the initialization phase satisfies
\[
\expt[n_i(T)] \le  \eta_i + \hat n_i(T), 
\] 
where $\eta_i = \max\{1, \lceil  \frac{4 \sigma_s^2}{\Delta_i^2}(2 \log K + 2a \log T) - \nu \rceil \} $, and 
\begin{multline*}
\hat n_i(T) = \max \Big\{
e^{ \frac{2 \beta_{i^*}^2 \subscr{\delta}{$i^*$-cond}^2 }{3a  \nu (1 + \subscr{\delta}{$i^*$-cond}^2)}}, e^{\frac{2 \beta_{i^*}^2}{3a}} 
\Big\}  \\ + \frac{3ac_1}{2(3ac_1 -4)}e^{\frac{c_2 \beta_{i}^2}{2}} +  e^{\frac{2 \beta_{i}^2}{3a}} 
+ \frac{3ac_1}{2(3ac_1 -4)}e^{\frac{c_2 \beta_{i}^2}{2}} . 
\end{multline*}
\end{theorem}
\begin{proof}
See~\ref{proof:corr-regret}.
\end{proof}

\begin{remark}[\bit{Regret of correlated UCL algorithm}]\label{rem:corr-regret} \textup{
Recall that the $n_i(T)$ in Theorem~\ref{thm:corr-regret} is the number of selections of a suboptimal arm $i$ after the initialization phase. For an uninformative prior, $\nu=1$ and each arm is selected once in the initialization phase. Consequently, the expression for $\eta_i$ will reduce to the expression in Theorem~\ref{thm:uncorr-regret}. In the expression for $\hat n_i(T)$ in Theorem~\ref{thm:corr-regret}, we consider only the worst case, which corresponds to the first case in~\eqref{eq:subopt-sel-extra}. Other cases can be considered in the spirit of~\eqref{eq:subopt-sel-extra}. However, the number of cases for a correlated prior will be significantly more than four, which is the number of cases for an uncorrelated prior. }

\textup{
The correlated UCL algorithm operates in two phases. The benefit of the correlation structure is most pronounced in the initialization phase: 
as mentioned in Remark~\ref{rem:corr-init},  a highly correlated prior helps reduce the number of initialization steps. Further, if the correlated prior is a true measure of the environment, then the upper bound on $n_i(T)$ will be small.  
However, the $\beta_i$s are large if such a highly correlated prior is not a true measure of the environment, or a high confidence is placed on the priors, i.e., the initial variances are small and the mean rewards in the prior are far from the true mean rewards at the arms. Large $\beta_i$s may lead to a large constant  in the upper bound on $n_i(T)$. 
} \oprocend
\end{remark}

\section{Numerical Illustrations} \label{sec:numerics}
In this section, we illustrate the results of the preceding two sections with data from numerical simulations. The theoretical results pertain to different quality priors defined by how rich is the information they can capture about the rewards associated with the bandit. Uninformative priors capture no information, while uncorrelated informative priors capture beliefs about individual arms. Correlated (informative) priors add to uncorrelated informative priors the ability to capture beliefs about the relationship between different arms, which we leverage in our new correlated UCL algorithm. When an informative prior models the environment well, we refer to it as a \emph{well-informed} prior; conversely, if the prior models the environment poorly, we refer to it as \emph{ill-informed}.

As in \cite{PR-VS-NEL:13d}, our simulations focus on the case of a spatially-embedded bandit problem, for which \cite{PR-VS-NEL:13d} showed that correlated priors can lead to higher performance. The simulations show that, among well-informed priors, those with richer information content result in higher performance. Theorems \ref{thm:uncorr-regret} and \ref{thm:corr-regret} allow us to quantify the extent to which a prior is well-informed.


We consider here the spatially-embedded bandit problem studied in \cite{PR-VS-NEL:13d}.  The reward surface is relatively smooth with regions of both high and low rewards. This means that  a correlated prior capturing length scale information can improve performance. The mean reward value is equal to 30, and the sampling variance for each arm is $\sigma_s^2 = 10$.

Figure \ref{fig:goodPriors} shows simulations from cases where the informative priors are well-informed. Mean cumulative regret computed from an ensemble of 100 simulations is shown for three priors: an uninformative prior, an informative uncorrelated prior, and an informative correlated prior. For all the simulations, the parameter $\epsilon$ was set equal to $1/\sqrt{10} \approx 0.316$, and for correlated priors the parameter $\nu$ was set equal to 1. The informative priors have an initial mean belief $\bs \mu_0$ with a higher value (equal to 100) in regions with high rewards, and a lower value of zero elsewhere. The uncorrelated prior sets $\sigma_0^2 = 10 = \sigma_s^2$, meaning the prior represents the equivalent of a single prior observation. The correlated prior sets $\sigma_i^2(0) = 10$ as in the uncorrelated case, and uses a correlation structure representing an exponential kernel as in \cite{PR-VS-NEL:13d}. This kernel encodes the information that the closer two arms are in the embedding space, the more correlated are their rewards.

The richer information provided by the informative priors results in better performance in this case where the priors are well-informed: the informative correlated prior results in less regret than the informative uncorrelated prior, which in turn results in less regret than the uninformative prior. For short horizons, the informative priors result in cumulative regret which is less than the Lai-Robbins lower bound. The UCL algorithm and the correlated UCL algorithm can violate the lower bound because of the additional information provided by the priors, which effectively shifts the regret curve leftwards. Asymptotically, however, the algorithms will tend to match the Lai-Robbins regret rate for any prior.

In contrast, Figure \ref{fig:badPriors} shows simulations from cases where the informative priors are variously ill-informed. Mean cumulative regret computed from an ensemble of $100$ simulations is shown for three increasingly informative priors, as in Figure \ref{fig:goodPriors}. The informative priors have an initial mean belief $\bs \mu_0$ that is uniform with each element $\mu_i^0 = 30$. As in Figure \ref{fig:goodPriors}, the uncorrelated prior sets $\sigma_0^2 = 10 = \sigma_s^2$, meaning the prior represents the equivalent of a single prior observation. The correlated prior sets $\sigma_i^2(0) = 10$ and uses a correlation structure that again represents an exponential kernel but with a longer length scale to represent a smoother reward surface.

Although the informative priors accurately represent the overall mean value of the reward surface, they fail to capture the spatial heterogeneity of the reward surface, in particular the fact that it has high- and low-value patches. Therefore, both informative priors are ill-informed about the mean rewards and the informative uncorrelated prior results in much poorer performance than the uninformative prior for moderate task horizons. However, by adding the correlation structure to the ill-informed uncorrelated prior, we can recover much of the performance exhibited by the well-informed correlated prior of Figure \ref{fig:goodPriors}. In a spatially-embedded task like the one studied here, information about correlation structure among arms can be as valuable as accurate information about the value of individual arms. 

\begin{figure}[ht!]
   \centering
   \includegraphics[width=0.5\textwidth]{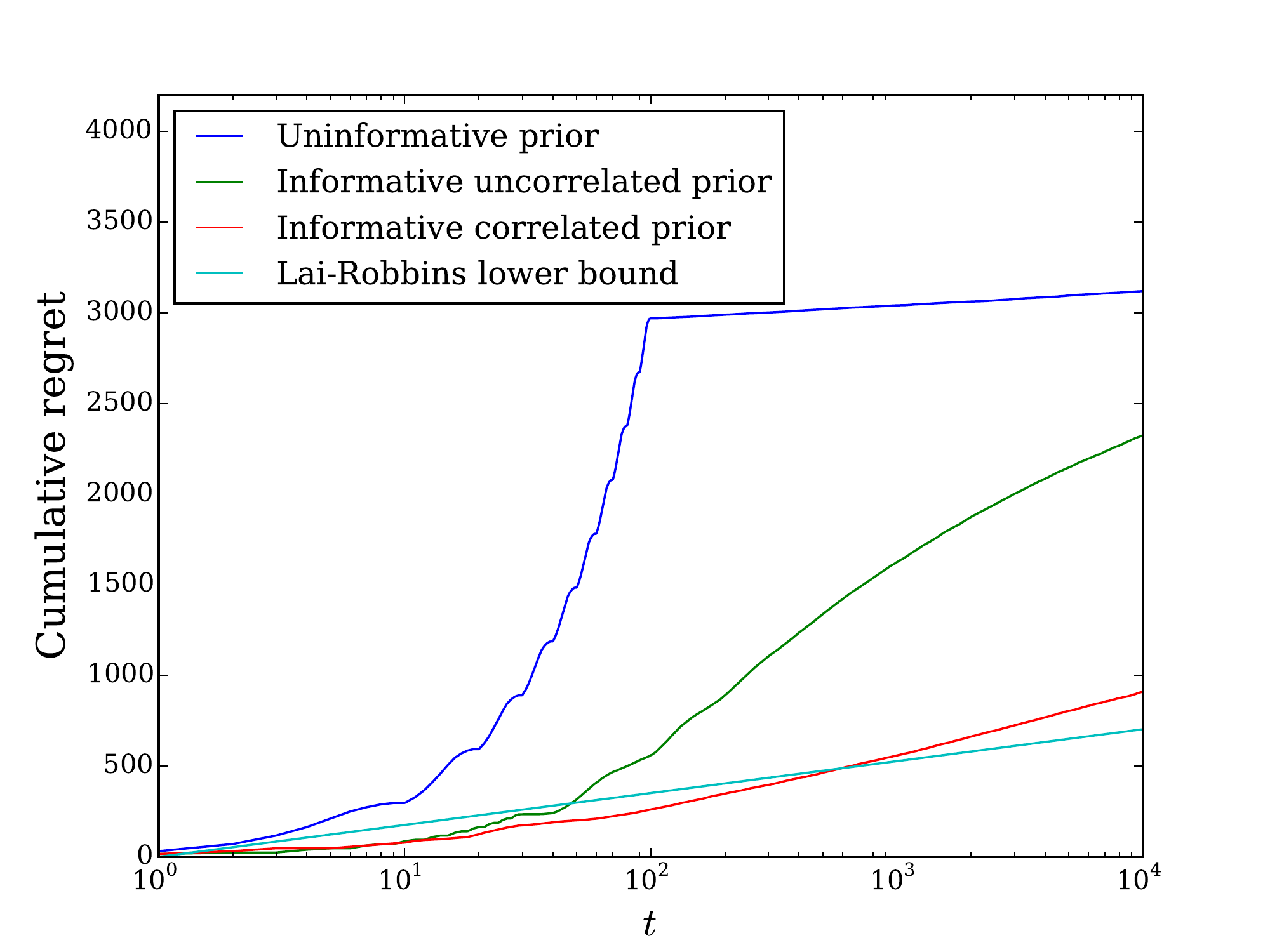} 
   \caption{Well-informed priors. Increasing the amount of information given increases performance. The traces show mean cumulative regret from 100 simulations for each of three different priors that model increasingly rich information about the rewards: the uninformative prior provides no information,  the informative uncorrelated prior provides information about rewards associated to individual arms, and the informative correlated prior adds information about the relationship between rewards associated with different arms. When used with an uninformative prior, the algorithm must begin by sampling each arm once in what is effectively an initialization phase. Upon completing this phase the algorithm can sample arms more selectively which makes the regret grow more slowly, as can be seen in the bend in the curve at $t=100$. Because of the additional information provided by the informative priors, the algorithms can sample arms more selectively from the initial time $t=1$, which results in better performance than the uninformative prior and allows the algorithms to outperform the Lai-Robbins bound on regret.}
   \label{fig:goodPriors}
\end{figure}

\begin{figure}[ht!]
   \centering
   \includegraphics[width=0.5\textwidth]{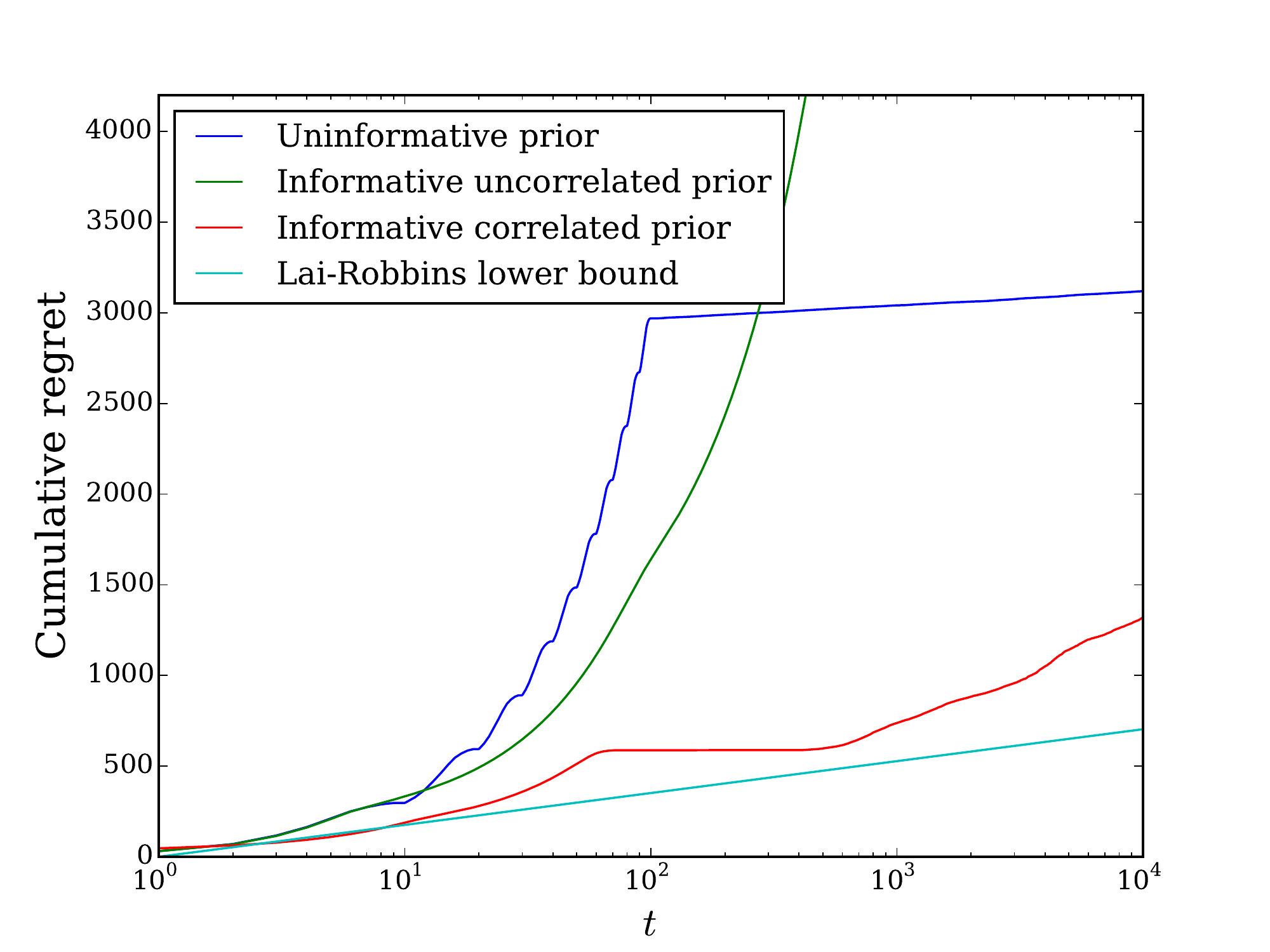} 
   \caption{Ill-informed priors. Increasing the amount of information given can decrease performance. As in Figure \ref{fig:goodPriors}, the traces show mean cumulative regret from 100 simulations for each of three different priors. Again the algorithms exhibit an initialization phase behavior for the uninformative and informative correlated priors, whose end can be seen in the bends in the regret curves near $t=100$. The ill-informed correlated prior improves performance relative to the uninformative prior although not quite as much as the well-informed correlated prior does in Figure \ref{fig:goodPriors}. In contrast, the ill-informed uncorrelated prior significantly decreases performance relative to all other priors. By encoding a strong incorrect belief about the rewards, this prior requires multiple samples of suboptimal arms to learn that they are suboptimal. This appears in the regret curve as an initialization phase that lasts until $t=$ 4,500, at which point the mean cumulative regret is approximately 35,000.}
   \label{fig:badPriors}
\end{figure}

\section{Conclusions and Future Directions} \label{sec:conclusions}
In this note we studied  and modified the UCL algorithm for the correlated MAB problem with Gaussian rewards. We investigated the influence of the assumptions in the prior on the performance of the UCL algorithm and the new correlated UCL algorithm. We characterized scenarios in which the informative priors perform better than the uninformative prior and characterized the improvement in the performance in terms of cumulative regret.   In particular, we showed conditions in which an informative correlated prior can be leveraged to significantly reduce cumulative regret.

There are several possible avenues of future research. First, we considered that the environment is stationary. An interesting future direction is to consider non-stationary environments in which the reward at each arm may be time-varying and the autocorrelation scale may be known. Second, we considered these problems for a single player. Many application scenarios involve a group of individuals and it is of interest to study collaborative and competitive multiplayer versions of these problems.




\appendix

\section{Proof of Theorem~\ref{thm:uncorr-regret}} \label{proof-uncorr-regret}

In the spirit of~\cite{PA-NCB-PF:02}, we bound $n_{i}(T)$ as follows:
\begin{align*}
n_{i}(T) &= \sum_{t=1}^T \indicator{i_t=i}\\
 &\leq \sum_{t = 1}^T \indicator{Q_{i}^t > Q_{i^*}^t}\\
  & \leq \eta_i + \sum_{t = 1}^T \indicator{Q_{i}^t > Q_{i^*}^t, n_{i}(t-1) \geq \eta_i},
 \end{align*}
where $\eta_i$ is some positive integer and $\indicator{x}$ is the indicator function, with $\indicator{x} = 1$ if $x$ is a true statement and $0$ otherwise.

At time $t$, the agent picks option $i$ over $i^*$ only if
\[ Q_{i^*}^t \leq Q_{i}^t.\]
This is true when at least one of the following equations holds:
\begin{align}
\mu_{i^*}(t) &\leq m_{i^*} - C_{i^*}(t) \label{eq:Qi*}\\
\mu_{i}(t) &\geq m_i + C_{i}(t)\label{eq:Qi}\\
m_{i^*} &< m_i + 2C_{i}(t)\label{eq:miComp}
\end{align}
where $C_{i}(t) = \frac{\sigma_s}{\sqrt{\delta^2 + n_{i}(t)}} \Phi^{-1}(1-\alpha_t)$ and $\alpha_t = 1/Kt^a$. Otherwise, if none of the equations (\ref{eq:Qi*})-(\ref{eq:miComp}) holds,
\begin{multline*}
Q_{i^*}(t) = \mu_{i^*}(t)+ C_{i^*}(t) > m_{i^*} \\
\geq m_i + 2C_{i}(t )> \mu_{i}(t) + C_{i}(t) = Q_{i}(t),
\end{multline*}
and option $i^*$ is picked over option $i$ at time t.
%

As noted earlier, the posterior mean $\mu_i(t)$ is a Gaussian random variable:
\[ 
\mu_{i}(t) \sim \mathcal{N}\left(\frac{\delta^2 \mu_{i}^0+ n_{i}(t)  m_i}{\delta^2 + n_{i}(t)}, \frac{n_{i}(t) \sigma_s^2}{(\delta^2 + n_{i}(t) )^2} \right).
\]
We will now analyze the events~\eqref{eq:Qi*},~\eqref{eq:Qi},~and~\eqref{eq:miComp}. Let $\prob_1(t)$  be the probability of the event~\eqref{eq:Qi*}.
\begin{lemma}[\bit{Probability of event~\eqref{eq:Qi*}}]\label{lem-first-event-uncorr}
The following statements hold for  event~\eqref{eq:Qi*}:
\begin{enumerate}
\item if $\Delta m_{i^*} \le 0 $, then 
\[
\sum_{t=1}^T \prob_1(t) \le \frac{a}{K(a-1)}.
\]
\item if $\Delta m_{i^*} > 0 $, then 
\begin{align*}
\sum_{t=1}^{T} \prob_1(t)  \le \max \Big\{ e^{\frac{2 \delta^4 \Delta m_{i^*}^2}{3 a \sigma_s^2 }}, e^{\frac{2 \Delta m_{i^*}^2}{3 a \sigma_0^2}} \Big\} 
+ \frac{3 a c_1}{2(3 a c_1 -4)} e^{ \frac{c_2 \delta^4 \Delta m_{i^*}^2}{2\sigma_s^2}}.
\end{align*}
\end{enumerate}
\end{lemma}
\begin{proof}
For $n_{i^*}(t) \ge 1$,  event~\eqref{eq:Qi*} is true if 
\begin{align*}
&m_{i^*} \geq \mu_{i^*}(t) + \frac{\sigma_s}{\sqrt{\delta^2 + n_{i}(t)}} \Phi^{-1}(1-\alpha_t)\\
\iff &m_{i^*}-\mu_{i^*}(t) \geq \frac{\sigma_s}{\sqrt{\delta^2 + n_{i}(t)}} \Phi^{-1}(1-\alpha_t)\\
\iff & z \leq -\sqrt{\frac{n_{i^*}(t) + \delta^2}{n_{i^*}(t)}}  \Phi^{-1}(1-\alpha_t) 
 + \frac{\delta^2}{\sigma_s} \frac{\Delta m_{i^*}}{\sqrt{n_{i^*}(t)}},
\end{align*}
where $z \sim \mcN(0,1)$ is a standard normal random variable. 

Similarly, for  $n_{i^*}(t) = 0$,  event~\eqref{eq:Qi*} is not true if (i) $\Delta m_{i^*} \le 0$, or (ii) $\Delta m_{i^*} >0 $ and $\Phi^{-1}(1- \alpha_t) \ge  {\Delta m_{i^*}}/ {\sigma_0}$.

We now establish the first statement. If $\Delta m_{i^*} \le 0$ and $n_{i^*}(t) =0$, then $\prob_1(t) =0$.
If $\Delta m_{i^*} \le 0$ and $n_{i^*}(t)  \ge 1$, then 
\begin{align*}
\prob_1(t) &\le  \prob \Big(  z \ge  \Phi^{-1}(1-\alpha_t) 
 - \frac{\delta^2 \Delta m_{i^*}}{\sigma_s} \Big) \\
 &\le  \prob (  z \ge  \Phi^{-1}(1-\alpha_t) ) = \alpha_t.
\end{align*}
Therefore,
\[
\sum_{t=1}^T \prob_1(t)  \le \sum_{t=1}^{+\infty} \frac{1}{Kt^a} \le \frac{1}{K} + \frac{1}{K(a-1)}= \frac{a}{K(a-1)}.
\]

To establish the second statement, we note that if $\Delta m_{i^*} > 0$ and $n_{i^*}(t) =0$, then event~\eqref{eq:Qi*} does not hold if 
\begin{align*}
\Phi^{-1}(1- \alpha_t) > \sqrt{\frac{3a}{2} \log t} \ge  \frac{\Delta m_{i^*}}{\sigma_0} 
\implies  t >   e^{2 \Delta m_{i^*}^2/3 a \sigma_0^2}.
\end{align*}
If $\Delta m_{i^*} > 0$ and $n_{i^*}(t)  \ge 1$, then $\prob_1(t) \le \prob(z \ge \zeta)$, where $\zeta =  \sqrt{\frac{3a}{2} \log t} - \frac{\delta^2 \Delta m_{i^*}}{\sigma_s}$. Note that $\zeta \ge 0$, if $t \ge e^{\frac{2 \delta^4 \Delta m_{i^*}^2}{3a \sigma_s^2}}$.
Define 
\[
t^{\dag}_1 = \max \Big\{ e^{\frac{2 \delta^4 \Delta m_{i^*}^2}{3 a \sigma_s^2 }}, e^{\frac{2 \Delta m_{i^*}^2}{3 a \sigma_0^2}} \Big\}.
\]
It follows that for $t \ge  t^{\dag}$, 
\begin{align*}
\prob_1(t)  & \le \frac{1}{2} e^{-\zeta^2/2} \\
& \le \frac{1}{2} \exp\Big( -\frac{1}{2} \Big( \sqrt{\frac{3a}{2} \log t} -  \frac{\delta^2 \Delta m_{i^*}}{\sigma_s}  \Big)^2  \Big) \\
&  \le \frac{1}{2} \exp\Big( -\frac{1}{2} \Big( \frac{3a c_1}{2} \log t - \frac{c_2 \delta^4 \Delta m_{i^*}^2}{\sigma_s^2} \Big)  \Big) \\
& = \frac{1}{2} e^{ \frac{c_2 \delta^4 \Delta m_{i^*}^2}{2\sigma_s^2}} t^{- \frac{3 a c_1}{4}},
\end{align*}
where the second last inequality follows from Lemma~\ref{lem:diff-of-squares} and $c_1$ and $c_2$ are as defined in Section \ref{sec:uncorr-UCL-regret}.

Therefore,
\begin{align*}
\sum_{t=1}^{T} \prob_1(t) & \le t^{\dag}_1 + \sum_{t=1}^{\infty}\frac{1}{2} e^{ \frac{c_2 \delta^4 \Delta m_{i^*}^2}{2\sigma_s^2}} t^{- \frac{3 a c_1}{4}} \\
& \le t_1^{\dag} + \frac{3 a c_1}{2(3 a c_1 -4)} e^{ \frac{c_2 \delta^4 \Delta m_{i^*}^2}{2\sigma_s^2}}.
\end{align*}
\end{proof}
 Let $\prob_2(t)$  be the joint probability of the event~\eqref{eq:Qi} and the event $n_i(t)>\eta_i$, for some $\eta_i \in \naturals$.
\begin{lemma}[\bit{Probability of event~\eqref{eq:Qi}}]\label{lem-sec-event-uncorr}
The following statements hold for  event~\eqref{eq:Qi}:
\begin{enumerate}
\item if $\Delta m_{i} < 0 $, then 
\begin{align*}
\sum_{t=1}^{T} \prob_2(t) &  \le e^{\frac{2 \delta^4 \Delta m_{i}^2}{3 a \sigma_s^2 \eta_i}}  +  \frac{3 a c_1}{2(3 a c_1 -4)}e^{\frac{c_2 \delta^4 \Delta m_{i}^2}{ 2\sigma_s^2 \eta_i} }.
\end{align*}
\item if $\Delta m_{i} \ge 0 $, then 
\[
\sum_{t=1}^T \prob_2(t)  \le  \frac{a}{K(a-1)}.
\]
\end{enumerate}
\end{lemma}
\begin{proof}
The event~(\ref{eq:Qi}) holds if
\begin{align*}
& m_i \leq \mu_{i}(t) - \frac{\sigma_s}{\sqrt{\delta^2 + n_{i}(t)}} \Phi^{-1}(1-\alpha_t)\\
\iff & \mu_{i}^t -m_i \geq \frac{\sigma_s}{\sqrt{\delta^2 + n_{i}(t)}} \Phi^{-1}(1-\alpha_t)\\
\iff & z \geq \sqrt{\frac{n_{i}(t) + \delta^2}{n_{i}(t)}} \Phi^{-1}(1-\alpha_t) + \frac{\delta^2}{\sigma_s} \frac{\Delta m_i}{\sqrt{n_{i}(t)}},
\end{align*}
where $z \sim \mcN(0,1)$ is a standard normal random variable.

We start with establishing the first statement. If $\Delta m_{i} < 0$ and $n_i(t) >\eta_i$, then 
\begin{align*}
\prob_2(t) & \le  \prob \Big(  z \ge  \Phi^{-1}(1-\alpha_t) 
 + \frac{\delta^2}{\sigma_s} \frac{\Delta m_{i}}{\sqrt{\eta_i}}  \Big) \\
 & \le  \prob (  z \le \zeta),
\end{align*}
where $\zeta =  \sqrt{\frac{3a}{2} \log t }
 + \frac{\delta^2}{\sigma_s} \frac{\Delta m_{i}}{\sqrt{\eta_i}}$.
 
It follows that $\zeta \ge 0$, if $t \ge  t^{\dag}_2 := e^{\frac{2 \delta^4 \Delta m_{i}^2}{3 a \sigma_s^2 \eta_i}}$.
It follows that for $t \ge  t^{\dag}_2$
\begin{align*}
\prob_2(t) & \le \frac{1}{2} e^{-\zeta^2/2} \\
& \le \frac{1}{2} \exp\Big( -\frac{1}{2} \Big( \sqrt{\frac{3a}{2} \log t} -  \frac{\delta^2}{\sigma_s} \frac{\Delta m_{i}}{\sqrt{\eta_i}} \Big)^2  \Big) \\
&\le \frac{1}{2} \exp\Big( -\frac{1}{2} \Big( \frac{3a c_1 \log t}{2} - \frac{c_2 \delta^4 \Delta m_{i}^2}{\sigma_s^2 \eta_i} \Big)  \Big) \\
&= \frac{1}{2} e^{\frac{c_2 \delta^4 \Delta m_{i}^2}{ 2\sigma_s^2 \eta_i} }t^{-\frac{3ac_1}{4}},
\end{align*}
where the second last inequality follows from Lemma~\ref{lem:diff-of-squares}. Therefore,
\begin{align*}
\sum_{t=1}^{T} \prob_1(t) & \le t^{\dag}_2 + \sum_{t=1}^{\infty}\frac{1}{2}e^{\frac{c_2 \delta^4 \Delta m_{i}^2}{ 2\sigma_s^2 \eta_i} }t^{-\frac{3ac_1}{4}} \\
& \le t^{\dag}_2  +  \frac{3 a c_1}{2(3 a c_1 -4)}e^{\frac{c_2 \delta^4 \Delta m_{i}^2}{ 2\sigma_s^2 \eta_i} }.
\end{align*}
The second statement follows similarly to the first statement in Lemma~\ref{lem-first-event-uncorr}.
\end{proof}

We now analyze the probability of event~\eqref{eq:miComp}. 
\begin{align}
&m_{i^*} < m_i + \frac{2 \sigma_s}{\sqrt{\delta^2 + n_{i}(t)}}\Phi^{-1}(1-\alpha_t) \nonumber\\
\iff & \Delta_i <  \frac{ 2 \sigma_s}{\sqrt{\delta^2 + n_{i}(t)}}\Phi^{-1}(1-\alpha_t) \nonumber\\
\implies & \frac{\Delta_i^2}{4 \sigma_s^2}(\delta^2 + n_{i}(t)) < - 2 \log \alpha_t \label{eq:third-condition}\\
\iff & \frac{\Delta_i^2}{4 \sigma_s^2}(\delta^2 + n_{i}(t)) < 2 \log K + 2a \log t \nonumber\\ 
\implies & \frac{\Delta_i^2}{4 \sigma_s^2}(\delta^2 + n_{i}(t)) < 2 \log K + 2a \log T \label{eq:monotonicity} 
\end{align}
where $\Delta_i = m_{i^*}-m_i$, the inequality~\eqref{eq:third-condition} follows from Lemma~\ref{lem:ineq}, and the inequality~\eqref{eq:monotonicity} follows from the monotonicity of the logarithmic function. 
Therefore,  the event~(\ref{eq:miComp}) is not true if  
\begin{align*}
n_{i}(t)  \ge \frac{4  \sigma_s^2}{\Delta_i^2}(2 \log K + 2a \log T) - \delta^2.
\end{align*}
Setting $\eta_i = \max\{1, \lceil  \frac{4 \sigma_s^2}{\Delta_i^2}(2 \log K + 2a \log T) - \delta^2\rceil \} $, we get
\begin{align*}
 \E{n_{i}^T} &\leq \eta_i + \sum_{t = 1}^T \prob(Q_{i}^t > Q_{i^*}^t, n_{i}(t-1) \geq \eta_i) \\
 & = \eta_i + \sum_{t = 1}^T \big(\prob_1(t) + \prob_2(t)\big)\\
 &<  \eta_i + \hat n_i(t). 
\end{align*}
This completes the proof of the theorem.

\section{Proof of Theorem~\ref{thm:corr-regret}} \label{proof:corr-regret}
Similar to the proof of Theorem~\ref{thm:uncorr-regret},
at time $t$, the agent picks option $i$ over $i^*$ only if $Q_{i^*}^t \leq Q_{i}^t$. 
This is true when at least one of the following equations holds:
\begin{align}
\mu_{i^*}(t) &\le m_{i^*} - C_{i^*}(t)  \label{eq:Qi*-corr}\\
\mu_{i}(t) &\ge m_i + C_{i}(t)   \label{eq:Qi-corr}\\
m_{i^*} &< m_i + 2C_{i}(t)   \label{eq:miComp-corr}
\end{align}
where $C_{i}(t) = \sigma_i(t) \sqrt{\sum_{j=1}^N {\rho_{ij}^2(t)}}  \Phi^{-1}(1-\alpha_t)$, $\alpha_t = 1/Kt^a$. 

For $n_i(t) \ge 1$ and $n_{i^*}(t) \ge 1$, equations~\eqref{eq:Qi*-corr}~and~\eqref{eq:Qi-corr} reduce to
\begin{align*}
z  & \ge  \frac {\sigma_{i^*}(t) \sqrt{\sum_{i=1}^N {\rho_{ij}^2(t)}}}{\bar \sigma_{i^*}(t)} \Phi^{-1}(1-\alpha_t) + \frac{e_{i^*}(t)}{\bar \sigma_{i^*}(t)}, \text{ and} \\
z  & \ge  \frac {\sigma_{i}(t) \sqrt{\sum_{i=1}^N {\rho_{ij}^2(t)}}}{\bar \sigma_{i}(t)} \Phi^{-1}(1-\alpha_t) - \frac{e_{i}(t)}{\bar \sigma_{i}(t)},
\end{align*}
respectively, where $e_i(t) = \sum_{j=1}^N \sum_{k=1}^N \sigma_{ik}(t) \lambda_{kj}^0 (\mu_0^j - m_j)$. 

\noindent
It follows that, for $n_{i^*}(t) \ge 1$,
\begin{align*}
\frac{|e_{i^*}(t)|}{\bar\sigma_{i^*}(t)} & \le  \frac{ \sigma_s \sum_{j=1}^N \sum_{k=1}^N \sigma_{i^*}(t) \sigma_k(t) |\lambda_{kj}^0| |\mu_0^j - m_j|}{\sqrt{n_{i^*}(t)}  \sigma_{i^*}^2 (t)} \\
& \le \frac{ \sigma_s^2 \sum_{j=1}^N \sum_{k=1}^N  |\lambda_{kj}^0| |\mu_0^j - m_j|}{ \sqrt{n_{i^*}(t) \nu}  \sigma_{i^*}(t)} \\ 
& \le \sigma_s \sqrt{ \frac{n_{i^*}(t) + \subscr{\delta}{$i^*$-cond}^2 }{ n_{i^*}(t) \nu} }  \sum_{j=1}^N \sum_{k=1}^N  |\lambda_{kj}^0| |\mu_0^j - m_j|\\
& \le \sigma_s \sqrt{ \frac{1 + \subscr{\delta}{$i^*$-cond}^2 }{ \nu} }  \sum_{j=1}^N \sum_{k=1}^N  |\lambda_{kj}^0| |\mu_0^j - m_j| = \beta_{i^*}.
\end{align*}
For $n_{i^*}(t) =0$, event~\eqref{eq:Qi*-corr} does not hold if
\begin{align*}
\sigma_{i^*}(t) \Phi^{-1}(1-\alpha_t) & \ge \subscr{\sigma}{$i^*$,cond}\sqrt{\frac{3a}{2}\log t} \\
& \ge \frac{\sigma_s^2}{\nu}\sum_{j=1}^N \sum_{k=1}^N  |\lambda_{kj}^0| |\mu_j^0-m_j| \\
& \ge |e_{i^*}(t)|.
\end{align*}
Thus, for $n_{i^*}(t)=0$, event~\eqref{eq:Qi*-corr} does not hold if 
\begin{align*}
t & \ge e^{ \frac{2 \beta_{i^*}^2 \subscr{\delta}{$i^*$-cond}^2}{\nu(1+ \subscr{\delta}{$i^*$-cond}^2)}}.
\end{align*}

It follows using the same argument as in Theorem~\ref{thm:uncorr-regret} that
\begin{multline*}
\sum_{t=1}^T \prob(\text{event}~\eqref{eq:Qi*-corr}) \le \max \Big\{
e^{ \frac{2 \beta_{i^*}^2 \subscr{\delta}{$i^*$-cond}^2}{\nu(1+ \subscr{\delta}{$i^*$-cond}^2)}}, e^{\frac{2 \beta_{i^*}^2}{3a}} 
\Big\} \\ 
+ \frac{3ac_1}{2(3ac_1 -4)}e^{\frac{c_2 \beta_{i^*}^2}{2}}.
\end{multline*}
Similarly, 
\begin{multline*}
\sum_{t=1}^T \prob(\text{event}~\eqref{eq:Qi-corr}, n_i(t) \ge 1) \le  e^{\frac{2 \beta_{i}^2}{3a}} 
+ \frac{3ac_1}{2(3ac_1 -4)}e^{\frac{c_2 \beta_{i}^2}{2}}.
\end{multline*}
Also, event~\eqref{eq:miComp-corr} is not true if 
\[
n_i(t) > \frac{4 \sigma_s^2 }{\Delta_i^2} (2 \log K + 2 a \log T ) - \nu.
\]
Adding the probabilities of the events~\eqref{eq:Qi*-corr}-\eqref{eq:miComp-corr}, we obtain the desired expression. 

\end{document}